\documentclass[12pt]{article}

\usepackage{url}
\usepackage{amsmath}
\usepackage{amssymb}
\usepackage{amsfonts}

\usepackage{enumitem}
\usepackage{amsthm}

\newcommand{\N}{\mathbb{N}}
\newcommand{\bbp}{\mathbb{P}}

\newcommand{\R}{\mathbb{R}}
\newcommand{\Z}{\mathbb{Z}}
\newcommand\iid{i.i.d.\ }

\theoremstyle{plain}
\newtheorem{theorem}{Theorem}
\newtheorem{lemma}[theorem]{Lemma}
\newtheorem{corollary}[theorem]{Corollary}
\newtheorem{proposition}[theorem]{Proposition}

\theoremstyle{definition}
\newtheorem{definition}[theorem]{Definition}
\newtheorem{example}[theorem]{Example}

\begin{document}

\title{The Mathematics of Benford's Law - A Primer}

\author{Arno Berger and Theodore P.~Hill}

\maketitle


\begin{abstract}
\noindent
This article provides a concise overview of the main mathematical
theory of Benford's law in a form accessible to scientists and
students who have had first courses in calculus and probability. In 
particular, one of the main objectives here is to aid researchers who 
are interested in applying Benford's law, and need to understand
general principles clarifying when to expect the appearance of 
Benford's law in real-life data and when not to expect it. A second 
main target audience is students of statistics or mathematics, at 
all levels, who are curious about the mathematics underlying this 
surprising and robust phenomenon, and may wish to delve more deeply
into the subject. This survey of the fundamental principles behind 
Benford's law includes many basic examples and theorems, but does 
not include the proofs or the most general statements of the theorems; 
rather it provides precise references where both may be found. 
\end{abstract}

\section{{Introduction}}\label{sec1}

Applications of the well-known statistical phenomenon called {\em
  Benford's law\/}, or {\em first-digit law}, have been increasing
dramatically in recent years. The online Benford database \cite{BerAHR09}, for
example, shows over 800 new entries in the past decade alone.  At the
{\it Cross-domain Conference on Benford's Law Applications\/} hosted 
by the Joint Research Centre of the European Commission in Stresa,
Italy in July 2019, organizers and participants both expressed a need
for a readily available and relatively non-technical summary of the 
mathematics underlying Benford's law. This article is an attempt to
satisfy that request. As such, this overview of the mathematics of 
Benford's Law is formulated without relying on more advanced concepts 
from such mathematical fields as measure theory and complex analysis.  

The topic of Benford's law has a rich and fascinating history. First
recorded in the 19th century, it is
now experiencing a wide variety of applications including detection 
of tax and voting fraud, analysis of digital images, and
identification of anomalies in medical, physical, and macroeconomic
data, among others. The interested reader is referred to
\cite{BerAH15,MilS15,NigM12B} for more extensive details on the 
history and applications of Benford's law. 

It is our hope that the present Benford primer will be useful for two
groups of readers in particular: First, researchers who are interested
in applying Benford's law, and need to understand general principles
clarifying when to expect the appearance of Benford's law in real-life
data, and when not to expect it; and second, science students at both 
the undergraduate and graduate levels who are curious about the
mathematical basis for this surprising phenomenon, and may wish to
delve more deeply into the subject and perhaps even try their hands at
solving some of the open problems. 

This survey includes special cases of most of the main Benford
theorems, and many concrete examples, but does not include proofs or
the most general statements of the theorems, most of which may be found as indicated in \cite{BerAH15}. 
The structure of the article is as follows: Section 2 contains the
notation and definitions; Section 3 the basic properties that
characterize Benford behavior; Section 4 the Benford properties of
sequences of constants; Section 5 the Benford properties of sequences
of random variables; and Section 6 a brief discussion of four common errors.

\section{Basic notation and definitions}\label{sec2}

In this survey, the emphasis is on {\em decimal\/}
representations of numbers, the classical setting of Benford's law, so
here and throughout $\log t$ means $\log_{10} t$, and all digits are
{\em decimal\/} digits. For other bases such as binary or hexadecimal,
analogous results hold with very little change, simply by replacing
$\log$ with $\log_b$ for the appropriate base $b$; the interested
reader is referred to \cite[p.\ 9]{BerAH15} for details.

Here and throughout, $\N = \{1, 2, 3,  \ldots\}$ denotes the positive
integers (or natural numbers), $\Z = \{\ldots, -2, -1, 0, 1, 2, \ldots\}$ the integers,
$\R = (-\infty, \infty)$ the real numbers, and $\R^+ = (0, \infty)$
the positive real numbers.  For real numbers $a$ and $b$, $[a, b)$
denotes the set (in fact, half-open interval) of all $x \in \R$ with $a \leq x < b$;
similarly for $(a, b], (a,b), [a,b]$.  
Every real number $x$ can be expressed uniquely as $x= \lfloor x
\rfloor + \langle x \rangle$, where $\lfloor x \rfloor$ and $ \langle
x \rangle$ denote the {\em integer part} and the {\em fractional part}
of $x$, respectively.  Formally, $\lfloor x \rfloor = \max \{k \in
\Z:k \leq x \}$ and $ \langle x \rangle = x - \lfloor x \rfloor$. For
example, $\lfloor 2 \rfloor = 2$ and $ \langle 2 \rangle = 0$, whereas
$\lfloor 10 \pi \rfloor = \lfloor 31.4 \ldots \rfloor = 31$ and
$\langle 10 \pi \rangle = 0.415 \ldots$.

The basic notion underlying Benford's law concerns the {\em leading
  significant digits\/} and, more generally, the {\em significand\/} of
a number (also sometimes referred to as the {\em mantissa\/} in
scientific notation). 

\begin{definition}\label{def21} 
For $x \in \R^+$, the ({\em decimal\/}) {\em significand\/} of $x$, denoted
$S(x)$, is given by $S(x) = t$, where $t$ is the unique number in $[1,
10)$ with $x = 10^k t$ for some (necessarily unique) $k \in \Z$. For
negative $x$, $S(x) = S(-x)$, and for convenience, $S(0)=0$.
\end{definition}

\begin{example}\label{ex22}
$S(2019) = 2.019 = S(0.02019) =S(-20.19) $.
\end{example}

\begin{definition}\label{def23}
The {\em first\/} ({\em decimal\/}) {\em significant digit\/} of $x\in \R$,
denoted $D{_1}(x)$, is the first (left-most) digit of $S(x)$, where by
convention the terminating decimal representation is used
if $S(x)$ has two decimal representations.
Similarly, $D{_2}(x)$ denotes the second digit of $S(x)$,  $D_3(x)$
the third digit of $S(x)$, and so on. (Note that $D_n(0) = 0$ for all
$n\in \N$.)
\end{definition}

\begin{example}\label{ex24}
$D{_1}(2019) = D{_1}(0.02019) = D{_1}(-20.19) = 2$, $D{_2}(2019) = 0$,
$D{_3}(2019) = 1$, $D{_4}(2019) = 9$, and $D_j(2019) = 0$ for all $j
\ge 5$. Also, $D_n(2019) \! = D_n (2018.9999\ldots)$ for all $n\in \N$.
\end{example}

As will be seen next, the formal notions of a Benford sequence of
numbers and a Benford random variable are defined via the
significands, or equivalently, via the significant digits of the
sequence and the random variable. 
An infinite sequence of real numbers $(x_1, x_2, x_3, \ldots)$ is
denoted by $(x_n)$;  e.g., $(2^n) = (2, 2^2, 2^3, \ldots) = (2, 4, 8,
\ldots)$.  In the next definition, $\# A $ denotes the number of
elements of the set $A$; e.g., $\#\{2, 0, 1, 9\} = 4$.

\begin{definition}\label{def28} 
A sequence of real numbers $(x_n)$ is a {\em Benford sequence},
or {\em Benford} for short, if for every $t \in [1,10)$, the limiting
proportion of $x_n$'s with significand less than or equal to $t$ is
exactly $\log t$, i.e., if
$$
\lim\nolimits_{N\to \infty} \frac{\# \{ 1 \le  n \le N : S(x_n) \le t \}}{N} = \log t 
\quad \mbox{\rm for all }  t \in [1,10).
$$ 
\end{definition}

\begin{example}\label{ex30}
(i) The sequence of positive integers $(n) = (1, 2, 3, \ldots)$ is
  not Benford, since, for example, more than half the entries less
  than $2 \cdot 10^m$ have first digit 1 for every positive integer
  $m$, so the limiting proportion of entries with significand less
  than or equal to 2, if it exists at all, cannot be $\log 2 <
  0.5$. Similarly, the sequence of prime numbers $(2, 3, 5, 7, 11,
  \ldots)$ is not Benford, but the demonstration of this fact is
  deeper; see \cite[Example 4.17(v)]{BerAH15}. 

\smallskip

\noindent
(ii) As will be seen in Example \ref{ex331} below, the sequences
  $(2^n)$ and $(3^n)$ of powers of 2 and 3 are Benford.  Many other
  classical sequences including the Fibonacci sequence $(1, 1, 2, 3,
  5, \ldots)$ and the sequence of factorials $(n!) = (1, 2, 6, 24,
  120, \ldots)$ are also Benford.
\end{example}

An equivalent description of a Benford sequence in terms of the
limiting proportions of values of its significant digits is as follows.

\begin{proposition}\label{prop29}
A sequence $(x_n)$ of real numbers is Benford if and only if 
\begin{align*}
\lim\nolimits_{N\to \infty} & \frac{ \# \{ 1 \le n \le N  : 
D_1(x_n) = d_1, D_2(x_n) = d_2,  \ldots, D_m(x_n) =  d_m \} }{N} = \\[2mm]
& \qquad = \log \left(  1  + \frac1{10^{m-1}d_1 + 10^{m-2} d_2+\ldots  + d_{m} }  \right),
\end{align*}
for  all $ m\in \N$, all $d_1\in \{1, 2, \ldots , 9\}$, and all
$d_j \in \{0,1, \ldots, 9\} $, $ j\ge2$.
\end{proposition}

\begin{example}
Proposition \ref{prop29} with $m=1$ yields  the well-known {\em
  first-digit law}: For every Benford sequence of real numbers $(x_n)$,
$$
\lim\nolimits_{N\to \infty} \! \frac{\# \{ 1\le n \le N : D_1(x_n) = d\}}{N}= 
\log \left(1 + \frac{1}{d}\right)
 \quad \!\! \! \mbox{\rm for all } d \in \! \{1,2, \ldots, 9\}.
$$
\end{example}

The notion of a Benford random variable (or dataset) is essentially
the same as that of a Benford sequence, with the limiting proportion
of entries replaced by the {\em probability\/} of the random values.

\begin{definition}\label{def25}
A (real-valued) random variable $X$ is {\em Benford\/} if 
$$
P(S( X ) \leq t) = \log t \quad \mbox{\rm for all } t \in [1, 10).
$$
\end{definition}

Recall that a random variable $U$ is said to be {\em uniformly
  distributed\/} on $[0,1]$ if $P(U \leq s) = s$ for all $s \in [0, 1]$. 

\begin{example}\label{ex354}
Let $U$ be uniformly distributed on $[0,1]$.

\smallskip

\noindent
(i) $U$ is not Benford, since as is easy to check, $P(S(U) \leq 2) =
  \frac19 < \log2$.

\smallskip

\noindent
(ii) $X = 10^U$ is Benford, since $S(X) = X$, and $P(S(X) \leq t) =
  P(X \leq t) = P(10^U \leq t) = P(U \leq \log t) = \log t$ for all $t
  \in [1, 10)$. In fact, this construction provides an excellent way
  of generating random data that follows Benford's law on a digital
  computer: Use any standard program to generate $U$, and then raise 10 to that power.
\end{example}

The analogous definition of a Benford random variable in terms of
significant digits follows similarly.

\begin{proposition}\label{prop26}
A random variable $X$ is Benford if and only if 
\begin{align*}
P\bigl(D_1(X) & = d_1, D_2 (X) = d_2 , \ldots , D_m(X) = d_m \bigr) = \\
& \qquad = \log \left(
  1  + \frac1{10^{m-1}d_1+ 10^{m-2} d_2 + \ldots + d_{m}}  \right),
\end{align*}
for all $ m\in \N$, all $d_1\in \{1, 2, \ldots , 9\}$, and all $d_j \in \{0,1, \ldots, 9\}$,  $j\ge2$.
\end{proposition}

\begin{example}\label{ex27}
If $X$ is a Benford random variable, then the probability that $X$ has the same first three digits as $\pi = 3.1415 \ldots$ is 
\begin{align*}
P\bigl( D_1(X) = 3, D_2(X) = 1, D_3(X) = 4 \bigr) & = \log \left( 1 +
  \frac1{10^2 \cdot 3 + 10 \cdot 1 + 4} \right)  \\
& =  \log
\frac{315}{314}  \approx 0.00138.
\end{align*}
\end{example}

None of the classical random variables are Benford exactly, although
some are close for certain values of their parameters.  For example,
no uniform, exponential, normal, or Pareto random variable is Benford
exactly, but Pareto and log normal random variables, among
others, can be arbitrarily close to being Benford depending on the
values of their parameters.

\section{{\bf What properties characterize Benford sequences and
    random variables?}}\label{sec3}

The purpose of this section is to exhibit several fundamental and
useful results concerning Benford sequences and random
variables. These include three basic properties of a sequence of
constants or a random variable that are equivalent to it being 
Benford:
\begin{enumerate}
\item the fractional parts of its decimal logarithm are uniformly
  distributed between 0 and 1;
\item the distribution of its significant digits is invariant under changes of scale; and 
\item the distribution of its significant digits is continuous and invariant under changes of base. 
\end{enumerate}
Analogous definitions and results also hold for {\em Benford
  functions}, for which the interested reader is referred to \cite[Section 3.2]{BerAH15}.

An additional feature demonstrating the robustness of Benford's law is
that if a Benford random variable is multiplied by any independent
positive random variable, then the product is Benford as well.

Recall that a sequence of real numbers $(x_n) = (x_1, x_2, x_3,
\ldots)$ is {\em uniformly distributed modulo one\/} (or {\em mod\/} 1, for short) if 
$$
\lim\nolimits_{N\to \infty}  \frac{\#\{1 \leq n \leq N: \langle x_n \rangle \leq s\}}{N} = s
\quad \mbox{\rm for all } s \in [0,1],
$$
e.g., in the limit, exactly half of the fractional parts $\langle x_n
\rangle$ are less than or equal to $\frac12$, and exactly one third are less than or equal to $\frac{1}{3}$.
The next lemma is a classical equidistribution theorem of Weyl, and,
as will be seen, is a powerful tool in Benford theory.

\begin{lemma}\label{lem341}
The sequence $(na) = (a, 2a, 3a, \ldots)$ is uniformly distributed mod~$1$ if and only if $a$ is irrational. 
\end{lemma}

\begin{proof}
See \cite[Proposition 4.6]{BerAH15}.  
\end{proof}

\noindent
The application of Lemma \ref{lem341} to the theory of Benford's law
is evident from the following basic characterization of Benford
sequences. (Here and throughout let $\log 0 = 0$ for convenience.)

\begin{theorem}\label{thm330}
A sequence of real numbers $(x_n)$ is Benford if and only if the
sequence $(\log |x_n|)  \! = \!  (\log |x_1|, \log |x_2|, \log |x_3|, \ldots)$ is uniformly distributed mod~$1$.
\end{theorem}

\begin{proof}
See \cite[Theorem 4.2]{BerAH15}. 
\end{proof}

\begin{example}\label{ex331}
(i) The sequence $(2^n)$ of powers of 2 is Benford.  This follows by
  Theorem \ref{thm330} and Lemma \ref{lem341} since $(\log 2^n) = (n
  \log 2)$ and since $\log 2$ is irrational. Similarly, the sequences
  $(3^n)$ and $(5^n)$ of powers of 3 and 5, respectively, are Benford.

\smallskip

\noindent
(ii) The sequence $(10^n)$ is not Benford, nor is
$\bigl(10^{n/2}\bigr) \! = \! \bigl(\sqrt{10}, 10, 10\sqrt{10},
  \ldots \bigr)$, since $\langle \log 10^{n/2} \rangle = \langle
  \frac{n}{2} \rangle = 0$ or $\frac12$ for every $n$, so $\bigl( \log
  10^{n/2} \bigr)$ is not uniformly distributed mod 1.
\end{example}

The following characterization of Benford random variables is a direct analogue of Theorem \ref{thm330}. 

\begin{theorem}\label{thm31}
A random variable $X$ is Benford if and only if the random variable
$\langle \log |X| \rangle$ is uniformly distributed on $[0, 1]$.
\end{theorem}

\begin{proof}
See \cite[Theorem 4.2]{BerAH15}. 
\end{proof}

The next proposition shows that if a sequence of numbers or a random
variable are Benford, then so are the positive multiples of the
sequence or random variable, as are their powers and reciprocals.

\begin{proposition}\label{prop332}
If the sequence of numbers $(x_n)$ is Benford, and if the random
variable $X$ is Benford, then for every $a > 0$ and $0 \neq k \in \Z$, the sequence $(ax_{n}^k)$ and the random variable $(aX^k)$ are also Benford.
\end{proposition}

\begin{proof}
Special case of  \cite[Theorem 4.4]{BerAH15}. 
\end{proof}

\begin{example}\label{ex333}
(i) Since $(2^n)$ is Benford, the sequences 
$(4^n)\! = \! (4, 16, 64, \ldots)$, $(2^{-n})  \! =\! (\frac{1}{2},
\frac{1}{4}, \frac{1}{8}, \ldots)$, and $(2^n \pi) = (2 \pi, 4 \pi, 8 \pi, \ldots)$ are also Benford.

\smallskip

\noindent
(ii) Since $X=10^U$ is Benford, so are $X^2=100^U, 1/X = 10^{-U}$,
  and $\pi X = \pi 10^U$.
\end{example}

The next theorem says that if a Benford random
variable is multiplied by any positive constant, e.g., as a result of
changing units of measurement, then the significant digit
probabilities will not change. In fact Benford random variables are
the only random variables with this property. Recall that two random
variables $X$ and $Y$ are {\em identically distributed\/} if $P(X\le
t) = P(Y \le t)$ for all $t\in \R$.

\begin{definition}\label{def33}
A random variable $X$ has {\em scale-invariant significant digits\/} if 
$S(X)$ and $S(aX)$ are identically distributed for all $a\in \R^+$.
\end{definition}

\begin{example}\label{ex355}
Let $U$ be uniformly distributed on $[0,1]$.

\smallskip

\noindent
(i) $U$ does not have scale-invariant digits since, for example,
  $P(S(U) \leq 2) = \frac19$ but $P(S(2U) \leq 2) = \frac59$.

\smallskip

\noindent
(ii) As is easy to check directly, or follows immediately from the next theorem and Example \ref{ex354} above, the random variable $X=10^U$ has scale-invariant significant digits.
\end{example}

\begin{theorem}\label{thm34}
A random variable $X$ with $P(X=0)=0$ is Benford if and only if it has scale-invariant significant digits.
\end{theorem}

\begin{proof}
See \cite[Theorem 5.3]{BerAH15}. 
\end{proof}

\begin{example}\label{exa35}
By Theorem \ref{thm34} and Example \ref{ex354} above,  if $U$ is
uniformly distributed on $[0, 1]$, then for every $a > 0$ the random variable $aU$
is not Benford, whereas the random variable $a10^U$ is Benford. \end{example}

In fact, a much weaker form of scale-invariance characterizes Benford's law completely, namely, scale-invariance of any single first digit.

\begin{theorem}\label{thm342}
A random variable $X$ with $P(X=0)=0$ is Benford if and only if for some $d \in \{1, 2, \ldots, 9\}$, 
$$
P(D_1(aX) = d) = P(D_1(X) = d) 
\quad \mbox {\rm for all } a \in \R^+ .
$$
\end{theorem}

\begin{proof}
See \cite[Theorem 5.8]{BerAH15}. 
\end{proof}

\begin{example}\label{ex360}
If $X$ is a positive random variable, and the probability that the first
significant digit of $aX$ equals 3 is the same for all $a\in \R^+$, then $X$ is Benford.
\end{example}

A notion parallel to that of scale-invariance is the notion of {\em
  base-invariance}, one interpretation of which says that the
distribution of the significant digits remains unchanged if the base
is changed from 10 to, say, 100.

\begin{definition}\label{def36}
A random variable $X$ has {\em base-invariant significant digits\/} if
$S(X)$ and $S(X^n)$ are identically distributed for all $n\in \N$.
\end{definition}

\begin{example}\label{ex356}
Let $U$ be uniformly distributed on $[0,1]$.

\smallskip

\noindent
(i) A short calculation (e.g., see \cite[Example 5.11(iii)]{BerAH15}) shows that $U$ does not have base-invariant significant digits.

\smallskip

\noindent
(ii) A random variable $Y$ with $P(S(Y) = 1) = 1$ clearly has
  base-invariant  significant digits, as does any Benford random
  variable, which follows by a short calculation; see \cite[Example 5.11(ii)]{BerAH15}.
\end{example}

As seen in the last example, random variables whose significand equals
$1$ with probability one, and Benford random variables both have
base-invariant significant digits. In fact, as the next theorem shows, averages of these two distributions are the only such random variables.

\begin{theorem}\label{thm38}
A random variable $Z$ with $P(Z = 0) = 0$ has base-invariant significant digits if and
only if $Z = (1-q)X + qY$ for some $q \in [0,1]$, 
where $X$ is Benford and $P(S(Y)=1)=1$.
\end{theorem}

\begin{proof}
See \cite[Theorem 5.13]{BerAH15}. 
\end{proof}

\begin{theorem}\label{cor343}
If a random variable has scale-invariant significant digits then it has base-invariant significant digits.
\end{theorem}

\begin{proof}
Follows immediately from Theorems \ref{thm34} and  \ref{thm38}.  
\end{proof}

A consequence of Theorem \ref{thm38} is that there are many
base-invariant random variables that are {\em not} Benford, but as the
next corollary shows, all continuous random variables that are base-invariant are also Benford.
Recall that a random variable $X$ is {\em continuous\/} if there exists a
function $f_X:\R \rightarrow [0, \infty)$, the {\em density
  function\/} of $X$, such that 
$$
P(X \leq t) = \int_{- \infty}^t f_X(x)\, {\rm d} x \quad \mbox {\rm for all  } t \in \R.
$$
As the reader may notice, such a random variable $X$ is often called
{\em absolutely continuous\/} in advanced texts, whereas the term
{\em continuous\/} refers to the (weaker) property that $P(X = t) =0$
for all $t\in \R$. In keeping with the elementary nature of this
article, random variables that have the latter property but not the
former (such as, e.g., Cantor random variables \cite[Example
8.9]{BerAH15}) are not considered here, and {\em continuous\/} means
{\em absolutely continuous\/} throughout. Many of the most common and
useful random variables are continuous, including uniform, normal, and
exponential random variables. Every Benford random variable is continuous.

\begin{corollary}\label{cor39}
A continuous random variable is Benford if and only if it has base-invariant significant digits.
\end{corollary}

The final theorem in this section illustrates one of the key
``attracting'' properties of Benford random variables, namely, if any
random variable is multiplied by an independent Benford random
variable, then the product is Benford.

\begin{theorem}\label{thm310}
Let $X, Y$ be independent random variables with ${P(XY \! = \! 0)=0}$.
If either $X$ or $Y$ is Benford, then the product $XY$ is also Benford.
\end{theorem}

\begin{proof}
See \cite[Theorem 8.12]{BerAH15}. 
\end{proof}

\begin{corollary}\label{cor311}
Let $X_1, X_2, \ldots$ be independent positive random variables.
If $X_j$ is Benford for some $j \in \N$, then the product $X_1X_2 \cdots X_m$  is Benford for all $m \geq j$.
\end{corollary}

\section{What sequences of constants are Benford?}\label{sec4}

The goal of this section is to describe the Benford behavior of
deterministic (that is, non-random) sequences. The sequences described below
will typically be increasing (or decreasing) sequences of positive
constants given by a rule that specifies the next entry in the
sequence as a function of the previous entry (or several previous entries, for
example, as in the Fibonacci sequence). The most common examples are
iterations of a single function, i.e., where the same function is
applied over and over again.
As will be seen here, three basic principles describe the Benford
behavior of such sequences: 
\begin{enumerate}
\item no polynomially increasing or decreasing sequence
  (or its reciprocals) is Benford; 
\item almost every, but not every, exponentially increasing positive
  sequence is Benford, and if it is Benford for one starting point, then it
  is Benford for all starting points; and 
\item every super-exponentially increasing or decreasing positive
  sequence is Benford for almost every, but not every, starting point.
\end{enumerate}
To facilitate discussion of iterations of a function $f:\R \rightarrow
\R$, the $n$th iterate of $f$ is denoted by $f^{[n]}$, so
$f^{[1]}(x) = f(x), f^{[2]}(x) = f\bigl( f(x)\bigr), f^{[3]}(x) =
f\bigl(f\bigl(f(x)\bigr) \bigr)$, etc.
Thus, $\bigl( f^{[n]}(x)\bigr)$ denotes the infinite sequence of
iterates of $f$ starting at $x$, i.e., 
$$
\bigl( f^{[n]}(x)\bigr) = \Bigl( f(x), f\bigl( f(x)\bigr),f\bigl( f
\bigl(f(x)\bigr) \bigr), \ldots \Bigr).
$$
The next example illustrates sequences with the three types of growth
mentioned above.

\begin{example}\label{ex319}
(i) Let $f(x) = x+1$.  Then $\bigl(f^{[n]}(x)\bigr) = (x+1, x+2,
  x+3, \ldots)$, so $\bigl( f^{[n]}(1)\bigr) = (2, 3, 4,  \ldots)$, a polynomially (in fact, linearly) increasing sequence.

\smallskip

\noindent
(ii) Let $g(x) = 2x$. Then $\bigl( g^{[n]}(x)\bigr) = (2x, 4x, 8x,
  \ldots)$, so $\bigl( g^{[n]}(1)\bigr) = (2, 4, 8, \ldots)$ and
  $\bigl( g^{[n]}(3)\bigr) = (6, 12, 24, \ldots)$, both exponentially increasing sequences.

\smallskip

\noindent
(iii) Let $h(x) = x^2$. Then $\bigl( h^{[n]}(x)\bigr) = (x^2, x^4,
  x^8, \ldots)$. Then $\bigl( h^{[n]}(1)\bigr) = (1, 1, 1, \ldots)$ is
  constant whereas $\bigl( h^{[n]}(2)\bigr) = (4, 16, 256, \ldots)$ is a super-ex\-po\-nen\-tially increasing sequence.
\end{example}

Recall from Example \ref{ex30}(i) that the sequence of positive
integers $(n)$ is not Benford.  Thus by the scale-invariance
characterization of Benford sequences in Theorem \ref{thm34} above, no
arithmetic sequence $(a, 2a, 3a, \ldots)$ is Benford for any real
number $a$ either.  In fact, no polynomially increasing sequence, or
the decreasing sequence of its reciprocals, is Benford. 

\begin{proposition}\label{prop353}
The sequence $(an^b) = (a, a2^b, a3^b, \ldots)$ is not Benford for any real numbers $a$ and $b$.
\end{proposition}

\begin{proof}
See \cite[Example 4.7(ii)]{BerAH15}. 
\end{proof}

\begin{example}
The sequences $(n^2) = (1, 4, 9,  \ldots)$ and $(n^{-2}) = (1, \frac{1}{4}, \frac{1}{9}, \ldots)$ are not Benford.
\end{example} 

Recall again that the sequence $(2^n)$ is Benford. This also follows
as a special case from the next theorem, which deals with
exponentially increasing sequences generated by iterations of linear
functions. Recall that a real number $a$ is a {\em rational power\/}
of $10$ if $a = 10^{m/k}$ for some $m, k \in \Z$, $k \neq 0$. For example, 
$\sqrt{10} = 10^{1/2}$ and $\sqrt[3]{100}=10^{2/3}$ are rational
powers of 10, but 2 and $\pi$ are not. As is easy to check, if $X$ is
a continuous random variable, then ${P(X \: \mbox{\rm  is a rational
    power of} \, \,  10) = 0}$.

\begin{theorem}\label{thm321}
Let $f(x) = ax + b$ for some real numbers $a > 1$ and $b \geq 0$.
Then for every $x>0$ the sequence $\bigl(f^{[n]}(x)\bigr)$ is Benford if and only if $a$ is not a rational power of $10$.
\end{theorem}

\begin{proof}
See \cite[Theorem 6.13]{BerAH15}. 
\end{proof}

\begin{example}\label{ex322}
(i) Let $f(x) = 2x$.  Since 2 is not a rational power of $10$, the
  sequence $\bigl(f^{[n]}(x)\bigr) = (2^nx)$ is Benford for every $x
  >0$; in particular taking $x=1$ shows that $(2^n)$ is Benford. Similarly, letting $g(x) = 2x + 1$, the sequence $\bigl(
  g^{[n]}(x)\bigr) = (2x + 1, 4x + 3, 8x + 7, \ldots )$ is also Benford for every $x > 0$.

\smallskip

\noindent
(ii) Let $g(x) =  \sqrt{10} x$.  Since $\sqrt{10} = 10^{1/2}$
  is a rational power of 10, the sequence $\bigl( g^{[n]}(x)\bigr) = (\sqrt{10} x, 10x, 10\sqrt{10} x, \ldots)$ is not Benford for any $x$. In particular, if $x=1$, the first significant digit of every entry in the sequence is either 1 or 3.
\end{example}
 
The Benford behavior of sequences generated by iterations of linear
functions as shown in Theorem \ref{thm321}, such as $(x_n)$ where
$x_{n+1}=2x_n+1$ for all $n>1$, has been extended to various wider
settings. One such setting is {\em linear difference equations}, where
the next entry in a sequence may depend linearly on several past
entries, such as the Fibonacci sequence $( 1,1,2,3,5,\ldots )$ where
$x_{n+1}=x_n+ x_{n-1}$; see \cite[Section 7.5]{BerAH15}.

As seen in Theorem \ref{thm321} above, for exponentially increasing
sequences generated by iterations of linear functions, the resulting
sequence is Benford or not Benford depending on the coefficient of the
leading term, and if it is Benford (or not Benford) for one starting
point $x > 0$, then it is Benford (not Benford, respectively) for {\em
  all\/} starting points $x >0$.
As will be seen in the next theorem, this is in contrast to the
situation for {\em super-exponentially\/} increasing (or decreasing)
functions, where the Benford property of the sequence $\bigl( f^{[n]}
(x)\bigr)$ does not depend on the coefficient of the leading term, but
does depend on the starting point $x$.  

\begin{theorem}\label{thm323}
Let $f$ be any non-linear polynomial with $f(x)>x$ for some real
number $a$ and all $x>a$.  Then
$\bigl( f^{[n]}(X)\bigr)$ is a Benford sequence with probability one
for every continuous random variable $X$ with $P(X>a)=1$, but there are infinitely
many $x > a$ for which $\bigl( f^{[n]}(x)\bigr)$ is not Benford.
\end{theorem}
 
\begin{proof}
See \cite[Theorem 6.23]{BerAH15}. 
\end{proof}

Thus super-exponentially increasing sequences are Benford for {\em
  almost all starting points\/} in the sense that if the starting
point is selected at random according to any continuous distribution
on $[a,\infty)$, then the resulting sequence is Benford with probability one.

\begin{example}\label{ex324}
(i) Let $f(x) = x^2 + 1$. Note that $f(x)>x$ for all $x$, so in Theorem
\ref{thm323} the number $a$ is arbitrary (or, more formally,
one may take $a=-\infty$). Thus there are infinitely many $x$
for which $\bigl( f^{[n]}(x)\bigr)$ is not Benford, but $\bigl(
f^{[n]}(X)\bigr)$ is Benford with probability one if $X$ is
continuous. However, in this example it is not easy to determine
exactly which starting points will yield Benford sequences. For instance,
it is unknown whether or not the sequence starting at $1$, i.e.,
$\bigl( f^{[n]}(1)\bigr) = (2, 5, 26, \ldots)$, is Benford; see \cite[Example 6.25]{BerAH15}.

\smallskip

\noindent
(ii) Let $g(x) = x^2$. Here Theorem \ref{thm323} applies with
$a=1$. Hence there are infinitely many $x >1$ so that $\bigl(
 g^{[n]}(x)\bigr) = (x^2, x^4, x^8, \ldots)$ is {\em not\/} Benford
 (e.g., $x = 10, 100, 1000, \ldots$). Since $g^{[n]}(1/x) =
 1/g^{[n]}(x)>0$ for all $n\in \N$ and $x\ne 0$, it follows with
 Proposition \ref{prop332} that if the starting point is selected at random via any continuous random variable $X$, then
 $\bigl( g^{[n]}(X)\bigr) = (X^2, X^4, X^8, \ldots)$ is Benford with probability one.
\end{example}
 
The results for iterations of functions above deal exclusively with
repeated application of the {\em same} function.  As another example
of the remarkable robustness of Benford's law, Benford sequences may also arise from the iterated application of {\em different\/} functions.
The next proposition, which follows easily from  \cite[Proposition 4.6(i)]{BerAH15} and Theorem \ref{thm330} above, provides an example of this behavior.

\begin{proposition}\label{prop325}
Let $f_1(x) = a_{1}x + b_1$ and  $f_2(x) = a_{2}x + b_2$ for some real
numbers $a_1, a_2 > 1$ and $b_1, b_2\ge 0$.
Letting $g_{n}(x) = f_{1}(x)$ if $n$ is odd, and $= f_{2}(x)$ if $n$
is even, then for every $x>0$ the sequence $\bigl( g^{[n]}(x)\bigr)
=\bigl( g_{1} (x), g_{2}\bigl( g_{1}(x) \bigr), \ldots\bigr)$ is
Benford if and only if $a_1a_2$ is not a rational power of $10$.
\end{proposition}

\begin{example}\label{ex326}
Alternating multiplication by 2 and by 3 yields a Benford sequence for
all starting points $x>0$.  In particular starting at $x=1$,  the sequence $(2, 6, 12, 36, 72, \ldots)$ is Benford.
\end{example}
 
In the last example, since iterations of each of the functions $f_1(x)
= 2x$ and $f_2(x) = 3x$ both lead to Benford sequences, it is perhaps
not surprising that alternating applications of them also leads to a
Benford sequence for every starting point $x>0$.  Similarly, even if
the selection of applying $f_1$ or $f_2$ is done {\em at random\/} by
flipping a fair coin at each step, the same conclusion holds (see
Example \ref{ex350} below). More surprisingly perhaps, even in
situations where $f_1$ on its own would not generate any Benford
sequences at all, and is applied more than half the time, the resulting
sequence $\bigl( g^{[n]} (x)\bigr)$ may still be Benford for most $x>0$.

\begin{example}\label{ex353}
Let $f_1(x) = \sqrt{x}$ and $f_2(x) = x^3$.  Then $\bigl(
f_1^{[n]}(x)\bigr)$ is not a Benford sequence for any $x>0$, since
$\bigl( f^{[n]}(x)\bigr) = (\sqrt{x}, \sqrt[4]{x}, \sqrt[8]{x},
\ldots)$ converges to $1$ as $n \rightarrow \infty$. By Theorem
\ref{thm323} and Proposition \ref{prop332}, on the other hand,  $\bigl( f_2^{[n]}(x)\bigr)$ is a
Benford sequence for almost all $x>0$.  As shown in \cite[Example
8.48]{BerAH15}, however, if the functions $f_1$ and $f_2$ are applied
randomly and independently at each step, with $f_1$ 
applied no more than $61.3$ percent of the time, then almost all of the
sequences generated are Benford. \end{example}

\section{ What sequences of random variables are Benford?}\label{sec5}

The goal of this section is to identify several of the key Benford
limiting properties of sequences of random variables. These include
the three basic facts that
\begin{enumerate}
\item powers of every continuous random variable converge to Benford's
  law; 
\item products of random samples from every continuous distribution
  converge to Benford's law; and 
\item 
if random samples are taken from random distributions that are chosen
in an unbiased way, then the combined sample converges to Benford's law.
\end{enumerate}
Here and throughout, {\em \iid}stands for {\em independent and
  identically distributed}; by definition, a {\em random sample} is a
finite sequence $X_1, X_2, \ldots, X_n$ of \iid  random variables.

\begin{definition}\label{def334}
An infinite sequence of random variables $(X_1, X_2,X_3 , \ldots$) {\em converges in distribution to Benford's law\/} if 
 $$
\lim\nolimits_{n\to \infty}  P(S( X_n ) \leq t) = \log t \quad \mbox{\rm for all } t\in [1,10),
$$
and {\em is Benford with probability one\/} if 
$$ 
P \bigl( (X_1, X_2, X_3, \ldots )   \: \mbox {\rm is a Benford
  sequence} \bigr)  = 1.
$$
\end{definition}

\noindent
In general, neither form of convergence implies the other, as the next example shows.

\begin{example}\label{ex351}
(i) Let $X$ be a Benford random variable, and for each $n \in \N$, let
$X_n = X$.  Then the sequence $(X_n) = (X, X, X, \ldots)$ converges to
Benford's law in distribution, since $P(S(X_n) \leq t)= \log t$ for
all $n$ and all $t \in [1,10)$. But $(X_n)$ is never a Benford
sequence, since no constant sequence is Benford.

\smallskip

\noindent
(ii) Let $X$ be a random variable that is identically 2, and let $X_n
= X^n$ for all $n \in \N$. Then $(X_n) = (2^n)$ is Benford with probability one
since $(2^n)$ is a Benford sequence. But for every $n \in \N$, $X_n =
2^n$ is constant, which implies, for example, that $P(D_1(X_n) = 1) =
0$ or 1, and hence does not converge to the Benford probability $\log 2$.  Thus the sequence $(X_n)$ does not converge in distribution to Benford's law. 

\smallskip

\noindent
(iii) If $X_1, X_2, \ldots$ are \iid random variables, then it is easy to see that the sequence $(X_n)$ converges in distribution to Benford's law if and only if it is Benford with probability one.
\end{example}
 
The next two theorems identify classical stochastic settings in which
sequential products of random variables converge in distribution to a
Benford  distribution, even though none of the random variables in the
product need be close to Benford at all. 

\begin{theorem}\label{thm312}
If $X$ is a continuous random variable, then $(X^n)$ converges in distribution to Benford's law and is Benford with probability one.
\end{theorem}
 
\begin{proof}
See  \cite[Theorem 8.8]{BerAH15}. 
\end{proof}

\begin{example}\label{ex357}
If $U$ is uniformly distributed on $[0,1]$, then by Example
\ref{ex354} above, $U$ is not Benford.  The sequence of random
variables $(U, U^2, U^3, \ldots )$, on the other hand, converges in
distribution to Benford's law and is Benford with probability one. In
fact, $(U^n)$ converges to Benford's law at rate $(n^{-1})$; see
 \cite[Figure 1.6]{BerAH15}.
 \end{example}
 
 As a complement to the last theorem, which shows that powers of every continuous random variable converge to Benford's law, the next theorem shows that products of random samples of every continuous random variable also converge to Benford's law.

\begin{theorem}\label{thm313}
If $X_1, X_2, \ldots$ are \iid continuous random variables, then the
sequence $(X_1, \! X_1X_2,\! \!$ $ X_1X_2X_3, \ldots )$ converges in distribution to Benford's law and is Benford with probability one.
\end{theorem}
 
\begin{proof}
See  \cite[Theorem 8.19]{BerAH15}. 
\end{proof}

\begin{example}\label{ex358}
If $U_1, U_2, \ldots$ are \iid random variables uniformly distributed
on $[0,1]$, then the sequence of products $U_1, U_1U_2, U_1U_2U_3,
\ldots$ converges to Benford's law in distribution and is Benford with
probability one. In fact, $(U_1 U_2 \cdots U_n)$ converges to Benford's
law at a rate faster than $(2^{-n})$; see \cite[Figure 8.3]{BerAH15}.
\end{example}

The next proposition illustrates a curious relationship between the Benford properties of powers of a single distribution and the products of random samples from that distribution.

\begin{proposition}\label{prop349}
Let $X_1, X_2, \ldots$ be \iid random variables.  If $(X_1\! ,
X_1^2\! ,  X_1^3\! , \ldots)$ is Benford with probability one, then so is
$ (X_1, X_1X_2 , X_1X_2X_3, \ldots)$.
\end{proposition}
 
\begin{proof}
See  \cite[Corollary 8.21]{BerAH15}. 
\end{proof}

\begin{example}\label{ex350}
Start with any positive number, and multiply repeatedly by either 2 or 3, where the multiplying factor each time is equally likely to be a 2 or a 3, and independent of the past.  The resulting sequence will be Benford with probability one. 

To see this, let $X_1, X_2, \ldots$ be \iid with $P(X_1 = 2) = P(X_1 = 3) = \frac12$.  Since the sequences $(2^n)$ and $(3^n)$ are both Benford, the sequence 
$(X_1^n) = (X_1, X_1^2, X_1^3, \ldots)$ is Benford with probability
one.  By Proposition \ref{prop349} this implies that the sequence
$(X_1, X_1 X_2, X_1 X_2 X_3, \ldots)$ is also Benford with probability
one, and since Benford sequences are scale-invariant for every $x>0$, the sequence $(xX_1$, $xX_1 X_2$, $xX_1 X_2 X_3, \ldots)$ is Benford with probability one.
\end{example}

Note that if $X_1, X_2, \ldots$ is a random sample from a distribution
that is not Benford, then the classical Glivenko--Cantelli Theorem
implies that the empirical distribution converges to the common
distribution of the $X_k$'s, which is {\em not\/} Benford.  On the
other hand, if random samples from {\em different} distributions are
taken in an ``unbiased'' way, then the empirical distribution of the combined sample will always converge to a Benford distribution. 
The final theorem in this section identifies a central-limit-like
theorem to model this type of convergence to a Benford distribution.
Intuitively, it says that when random samples (or data) from different
distributions are combined, then, if the different distributions are
chosen in an unbiased way, the resulting combined sample will converge to a Benford distribution.

\begin{definition}\label{def314}
A {\em random probability measure} $\bbp$ is a random variable whose values are probability measures on $\R$.
\end{definition}

\begin{example}\label{ex315}
(i) For a practical realization of a random probability measure
$\bbp$, simply roll a fair die --- if the die comes up 1 or 2, $\bbp$
is uniformly distributed on $[0,1]$, and otherwise $\bbp$ is
exponential with mean $1$. More formally, let 
$X$ be a random variable taking values in $\{1, 2, 3, 4, 5, 6\}$ with probability $\frac16$ each (e.g., the results of one toss of a fair die).
Let $P_1$ be uniformly distributed on $[0,1]$, and let $P_2$ be
exponentially distributed with mean $1$, i.e., $P_2\bigl( (-\infty, t]\bigr) = 1 - e^{-t}$ for all $t \geq 0$.  
Define the random probability measure $\bbp$ by $\bbp = P_1$ if $X =$
1 or 2, and $\bbp = P_2$ otherwise. Then with probability $\frac13$,
the value of $\bbp$ is a probability measure that is uniformly
distributed on $[0,1]$, and otherwise (i.e., with probability
$\frac23$), it is a probability measure in $\R$ that is exponential
with mean $1$; see  \cite[Example 8.33]{BerAH15}.  

\smallskip

\noindent
(ii) The classical iterative construction of a random cumulative
distribution function by Dubins and Freedman \cite{DubLF67R}  defines
a random probability measure $\bbp_{\rm DF}$; see  \cite[Example 8.34]{BerAH15}.
\end{example}
 
Clearly, some random probability measures will not generate Benford
behavior.  For example, if $\bbp$ is $P_1$ half the time and $P_2$
half the time, where $P_1$ is uniformly distributed on $[2,3]$ and
$P_2$ is uniformly distributed on $[4,5]$, then random samples from
$\bbp$ will not have any entries with first significant digit 1, and
hence cannot be close to Benford.

On the other hand, if a random probability measure is {\em unbiased\/} 
in a sense now to be defined, then it will always lead to Benford
behavior. The definition of unbiased below is based on the expected value of $\bbp$, that is, the single probability measure that is the average value of $\bbp$.
Given a random probability measure $\bbp$ and any $t \in \R$, the
quantity $\bbp\bigl( (-\infty, t] \bigr)$ is a random variable with
values between $0$ and $1$; denote its expected (average) value by
$E_{\bbp}(t)$.  It is easy to check that $E_{\bbp}(t)$ defines (or
more precisely, is the cumulative distribution function of) a
probability measure $P_{\bbp}$ on $\R$, {\em the average probability
  measure\/} of $\bbp$.

\begin{example}\label{ex335}
Let $\bbp$ be the random probability measure in Example
\ref{ex315}(i).  Then the average probability measure $P_{\bbp}$ 
is the probability distribution of a continuous random variable $X$ with density function $\frac{1}{3} + \frac{2}{3}e^{-x}$ for $0 < x < 1$ and $\frac{2}{3}e^{-x}$ for $x > 1$.
\end{example}
 
\begin{definition}\label{def336}
A random probability measure $\bbp$ has {\em scale-unbiased significant digits\/} if its average probability measure $P_{\bbp}$
has scale-invariant significant digits, and has {\em base-unbiased significant digits\/} if $P_\bbp$ has base-invariant significant digits.
\end{definition}
 
\begin{example}\label{ex359}
The classical Dubins-Freedman construction $\bbp_{\rm DF}$ mentioned
in Example \ref{ex315}(ii) above has both scale- and base-unbiased
significant digits; see  \cite[Example 8.46]{BerAH15}. 
\end{example}
 
The next theorem is the key result that shows that if random samples
are taken from distributions that are chosen at random in any manner
that is unbiased with respect to scale or base, then the resulting
empirical distribution of the combined sample always converges in
distribution to Benford's law. This may help explain, for example, why
the original dataset that Benford drew from many different sources,
why numbers selected at random from newspapers, and why experiments
designed to estimate the distribution of leading digits of all numbers
on the World Wide Web, all yield results that are close to the
logarithmic significant-digit law, i.e., Benford's law.

\begin{theorem}\label{thm339}
Let $\bbp$ be a random probability measure so that $\bbp (S \in
\{0,1\})=0$ with probability one. Let $P_1, P_2, \ldots$ be a random sample (\iid sequence) of probability measures from $\bbp$.  
Fix a positive integer $m$, and let $X_1, X_2, \ldots, X_m$ be a
random sample of size $m$ from $P_1$, let $X_{m+1}$, $\ldots, X_{2m}$
be a random sample of size $m$ from $P_2$, and so on. If $\bbp$ has
scale- or base-unbiased significant digits, then the empirical distribution of the combined sample $X_1, X_2, \ldots, X_m$, $X_{m+1}, \ldots$ converges to Benford's law with probability one, that is, 
$$
P\left(\lim\nolimits_{N \to \infty}  
 \frac{\#\{1 \leq n \leq N : S(X_n) \leq t\}}{N} = \log t \enspace \mbox{\rm
   for all } t \in [1, 10)  \right) = 1 .
$$
\end{theorem}

\begin{proof}
See \cite[Theorem 8.44]{BerAH15}, noting that slightly different
assumptions and notations are used there. 
\end{proof}

\begin{example}\label{ex361}
Since the classical Dubins-Freedman construction $\bbp_{\rm DF}$ has
scale- and base-unbiased significant digits (and has, with probability
one, no atoms), by Theorem \ref{thm339} above, combining random
samples from random distributions generated by $\bbp_{\rm DF}$ will guarantee that the empirical distribution of the combined sample converges to Benford's law.
\end{example}

\section{{\bf Common Errors}}\label{sec6}

The purpose of this section is to familiarize the reader with several recurring errors in the literature on Benford's law, in order that they may be avoided in future research and applications.

\bigskip

\noindent
{\bf  {\em Error 1. To be Benford, a random variable or dataset needs to cover at least several orders of magnitude.}}

\medskip

As seen in Example \ref{ex354}(ii), if $U$ is uniformly distributed on $[0,1]$, then $X=10^U$ is exactly Benford, yet $X$ takes only values between $1$ and $10$.

\bigskip

\noindent
{\bf  {\em Error 2.  Exponential sequences $(a^n) = (a, a^2,  a^3, \ldots)$ can generally be assumed to be Benford.}}

\medskip

As seen in Example \ref{ex331}, some exponentially increasing
sequences such as $(2^n)$ are Benford, and some such as
$(10^{n/2})$ are not, so care is needed. Even sequences
$(a^n)$ where $a$ is a rational power of $10$, although never Benford
exactly, may be very close to being Benford depending on $a$, as can be seen by looking at the sequence $(10^{n/100})$, since $(\langle \frac{n}{100} \rangle)$ is clearly close to being uniformly distributed on $[0,1]$. 

On the other hand {\em most} exponential sequences are Benford
in the sense that if the base number $x$ is selected at random via any
continuous distribution, then the sequence $(x^n)$ is Benford with
certainty (see Theorem \ref{thm312}), i.e., with probability one.

In contrast to this exponential case, no sequence $(na) = (a, 2a, 3a, 4a, \ldots)$ is Benford. Similarly, sequences of sums of i.i.d. random variables with finite variance are {\em never} Benford, as shown in \cite[Theorem 8.30]{BerAH15}. The authors conjecture that the restriction to distributions with finite variance is not necessary, and that 
``perhaps even no random walk on the real line at all has Benford paths (in distribution or with probability one)"  \cite[p.~200]{BerAH15}.

\bigskip

\noindent
{\bf {\em Error 3.  If a distribution or dataset has large spread and
    is regular, then it is close to Benford.}}

\medskip

Unfortunately, this error continues to be widely propagated, likely
because it may be traced back to the classical probability text of
Feller; see \cite{BerAH11B}. As the next example shows, this conclusion does not even hold for the ubiquitous and fundamental normal distribution.

\noindent

\begin{example}\label{ex347}
If $X = N(7, 1)$ then $P(D_1(X) = 1) \leq 0.000001$, so $X$ is not
close to being Benford.  Here $X$ is ``regular'' or ``smooth'' by
almost any criterion, and has standard deviation $1$, which may or may
not fit the criteria of having a ``large spread''.  On the other hand,
$Y = 100X$ is also regular and has much larger standard deviation than
$X$, but clearly $P(D_1(X) = 1) = P(D_1(Y) = 1)$, so $Y$ is also far
from being Benford.
\end{example}

Similarly, no uniform distribution is close to being Benford no matter how spread out it is, and in this case a universal discrepancy between uniform and Benford can be quantified.

\begin{example}\label{ex348}
No uniform random variable is close to Benford's law.  In particular, by \cite[Theorem 5.1]{BerAT18}, if $X$ is a uniform random variable, i.e.,  $X$ is uniformly distributed on $[a, b]$ for some $a<b$, then for some $1<t<10$,
$$
\lvert P(S(X) \leq t) - \log t \rvert \ge  0.0758 \ldots \, ;
$$
if $X\ge 0$ or $X\le 0$ with probability one then the (sharp) numerical bound on the
right is even larger, namely $0.134\ldots $.
\end{example}

Similar bounds away from Benford's law exist for normal and
exponential distributions, for example, but for these distributions
the corresponding sharp bounds are unknown \cite[p.~40]{BerAH15}.

\bigskip

\noindent
{\bf  {\em Error 4. There are relatively simple intuitive arguments to explain Benford's law in general.}}

\medskip

For some settings, such as exponentially increasing sequences of
constants, fairly simple arguments can be given to show when a
sequence is Benford, as was seen in Theorem \ref{thm330}. On the other
hand, there is currently no simple intuitive argument to explain the
appearance of Benford's law in the wide array of contexts in which it has been
observed, including statistics, number theory, dynamical systems, and
real-world data. More concretely, there is no theory at all, let alone
a simple one, even to decide whether the sequence $(1,2,5,26,677,
\ldots)$ starting with $1$ and proceeding by squaring the last number
and adding $1$, is Benford or not; see  Example \ref{ex324}(i). 
The interested reader is referred to \cite {BerAH11B} for a more detailed treatise on the difficulty of finding an easy explanation of Benford's law.


\section*{Acknowledgements}
The first author was partially supported by an NSERC Discovery Grant. 
Both authors are grateful to the Joint Research Centre of the European
Commission for the invitation to speak at their {\em Cross-domain conference on Benford's Law Applications\/} in Stresa, Italy in July
2019, and especially to the organizers of that conference, Professors 
Domenico Perrota, Andrea Cerioli, and Lucio Barabesi for their warm
hospitality.

\end{document}